\documentclass[11pt]{article}

\usepackage{amssymb, amsmath, amsthm, graphicx}
\usepackage[left=1in,top=1in,right=1in]{geometry}

\usepackage{subfigure}

\date{}

\theoremstyle{plain}
      \newtheorem{theorem}{Theorem}[section]
      \newtheorem{lemma}[theorem]{Lemma}
      
      \newtheorem{observation}[theorem]{Observation}
      \newtheorem{corollary}[theorem]{Corollary}
      
      \newtheorem{conjecture}[theorem]{Conjecture}
\theoremstyle{definition}

\theoremstyle{remark}

\def\hom{\mbox{\rm hom}}
\def\twr{\mbox{\rm twr}}
\def\rem{\mbox{\rm rem}}

\title{Ramsey-type results for semi-algebraic relations}

\author{David Conlon\thanks{Mathematical Institute, Oxford OX1 3LB, UK. Supported by a Royal Society University Research Fellowship. Email: {\tt 
david.conlon@maths.ox.ac.uk}.} \and Jacob Fox \thanks{Massachusetts Institute of Technology, Cambridge, MA. Supported by a Simons Fellowship, NSF grant DMS 
1069197, and by an MIT NEC Corporation Award. Email: {\tt fox@math.mit.edu}} \and J\'anos Pach\thanks{EPFL, Lausanne and Courant Institute, New York, NY. Supported 
by Hungarian Science Foundation EuroGIGA Grant OTKA NN 102029, by Swiss National Science Foundation Grant 200021-125287/1, and by NSF Grant CCF-08-30272. Email: 
{\tt pach@cims.nyu.edu}.}\and Benny Sudakov\thanks{Department of Mathematics, UCLA, Los Angeles, CA 90095. 
Research supported in part by NSF grant DMS-1101185, by AFOSR MURI grant FA9550-10-1-0569 and by a USA-Israel BSF grant. Email: {\tt bsudakov@math.ucla.edu}.} 
\and Andrew Suk\thanks{Massachusetts Institute of Technology, Cambridge, MA. Supported by an NSF Postdoctoral 
Fellowship and by Swiss National Science Foundation Grant 200021-125287/1. Email: {\tt asuk@math.mit.edu}.} }

\begin{document}

\maketitle

\begin{abstract} A $k$-ary semi-algebraic relation $E$ on $\mathbb{R}^d$ is a subset of $\mathbb{R}^{kd}$, the set of $k$-tuples of points in $\mathbb{R}^d$, 
which is determined by a finite number of polynomial equations and inequalities in $kd$ real variables. The description complexity of such a relation is at most 
$t$ if the number of polynomials and their degrees are all bounded by $t$. A set $A\subset \mathbb{R}^d$ is called homogeneous if all or none of the $k$-tuples 
from $A$ satisfy $E$. A large number of geometric Ramsey-type problems and results can be formulated as questions about finding large homogeneous subsets of sets 
in $\mathbb{R}^d$ equipped with semi-algebraic relations.

In this paper we study Ramsey numbers for $k$-ary semi-algebraic relations of bounded complexity and give matching upper and lower bounds, showing that they grow 
as a tower of height $k-1$.  This improves on a direct application of Ramsey's theorem by one exponential and extends a result of Alon, Pach, Pinchasi, Radoi\v 
ci\'c, and Sharir, who proved this for $k=2$. We apply our results to obtain new estimates for some geometric Ramsey-type problems relating to order types and 
one-sided sets of hyperplanes. We also study the off-diagonal case, achieving some partial results.
\end{abstract}

\section{Introduction}

{\bf Background and motivation.}
The term ``Ramsey theory" refers to a large body of deep results in mathematics which have a common theme: ``Every large system contains a large well organized 
subsystem."  This is an area in which a great variety of techniques from many branches of mathematics are used and whose results are important not only to graph 
theory and combinatorics but also to logic, analysis, number theory, computer science, and geometry.

The \emph{Ramsey number} $R(n)$ is the least integer $N$ such that every red-blue coloring of the edges of the complete graph $K_N$ on $N$ vertices 
contains a monochromatic complete subgraph $K_n$ on $n$ vertices.  Ramsey's theorem \cite{R30}, in its simplest form, states that $R(n)$ exists for every $n$. 
Celebrated results of Erd\H{o}s \cite{erdos2} and Erd\H os and Szekeres \cite{erdos} imply that $2^{n/2} < R(n) < 2^{2n}$ for every integer $n>2$. Despite much 
attention over the last 65 years (see, e.g., \cite{david2}), the constant factors in the exponents have not been improved.

Although already for graph Ramsey numbers there are significant gaps between the lower and upper bounds, our knowledge of hypergraph Ramsey numbers is even weaker.  The Ramsey number $R_k(n)$ is the minimum $N$ such that every red-blue coloring of all unordered $k$-tuples of an $N$-element set contains a monochromatic subset of size $n$, where a set is called monochromatic if all its $k$-tuples have the same color.  Erd\H os, Hajnal, and Rado \cite{erdos3, rado} showed that there are positive constants $c$ and $c'$ such that

$$2^{cn^2} < R_3(n) < 2^{2^{c'n}}.$$

\noindent
They also conjectured that $R_3(n) > 2^{2^{cn}}$ for some constant $c > 0$ and Erd\H os offered a \$500 reward for a proof.  For $k\geq 4$, there is also a difference of one exponential between the known upper and lower bounds for $R_k(n)$, namely,

$$\twr_{k-1}(cn^2) \leq R_k(n) \leq \twr_k(c'n),$$

\noindent where the tower function $\twr_k(x)$ is defined by $\twr_1(x) = x$ and $\twr_{i + 1} = 2^{\twr_i(x)}$.

The study of $R_3(n)$ is particularly important for our understanding of hypergraph Ramsey numbers. Given any lower bound on $R_k(n)$ for $k\ge 3$, an ingenious construction of Erd\H os and Hajnal, called the {\em stepping-up lemma} (see \cite{conlon,conlon2,graham}), allows us to give a lower bound on $R_{k+1}(n)$ which is exponentially larger than the one on $R_k(n)$. In the other direction, Erd\H os and Rado \cite{rado} came up with a technique that gives an upper bound on $R_{k + 1}(n)$ which is exponential in a power of $R_k(n)$.  Therefore, closing the gap between the upper and lower bounds for $R_3(n)$ would also close the gap between the upper and lower bounds for $R_k(n)$ for all $k$.  There is some evidence that the growth rate of $R_3(n)$ is indeed double exponential: Erd\H os and Hajnal (see \cite{conlon,graham}) constructed a 4-coloring of the triples of the set $[2^{2^{cn}}]$ which does not contain a monochromatic subset of size $n$. This result is best possible up to the value of the constant  $c$.

Ramsey numbers were first used by Erd\H os and Szekeres to give a bound on a beautiful geometric question which asks
for the smallest integer $ES(n)$ such that any set of $ES(n)$ points in the plane in general position contains $n$ elements in convex position, that is, $n$ 
points that form the vertex set of a convex polygon? The following argument due to Tarsi shows that

$$ES(n) \leq R_3(n)\leq 2^{2^{c' n}}.$$

\noindent Indeed, let $P = \{p_1,...,p_N\}$ be an ordered set of $N=R_3(n)$ points in the plane in general position. Color a triple $(p_i,p_j,p_k)$ red if $p_i,p_j,p_k$ appear in clockwise order along the boundary of their convex hull and color it blue otherwise.  By the choice of $N$, there exists a subset $S\subset P$ of size $n$ that is monochromatic.  It is easy to see that this monochromatic subset $S$ must be in convex position.  However, since the coloring on the triples of $P$ is defined ``algebraically", one might expect that this bound is not tight.  Indeed, a much stronger bound on $ES(n)$ was obtained by Erd\H os and Szekeres \cite{erdos}: they showed that $ES(n) \leq 2^{2n}$.

The main result of the present paper is an exponential improvement on the upper bound for hypergraph Ramsey numbers for colorings defined algebraically (a precise 
definition is given below). In particular, this shows that the Tarsi argument for estimating $ES(n)$ discussed above naturally results in an 
exponential bound. We also give a construction which implies an almost matching lower bound. The proofs of these results for $k$-uniform hypergraphs are based on 
adaptations of the Erd\H os-Rado upper bound technique and the Erd\H os-Hajnal stepping-up method. The key step, when reducing the problem from $(k+1)$-uniform 
hypergraphs to $k$-uniform ones, is to ensure that the algebraic properties of the underlying relations may be preserved.

\medskip

\noindent
{\bf Ramsey numbers for semi-algebraic relations.} 
A Boolean function $\Phi(X_1,X_2,...,X_t)$ is an arbitrary mapping from variables $X_1,...,X_t$,
attaining values ``true" or ``false", to $\{0,1\}$.
A set $A\subset \mathbb{R}^d$ is \emph{semi-algebraic} if there are polynomials $f_1,f_2,...,f_t \in \mathbb{R}[x_1,...,x_d]$ and a Boolean function
$\Phi(X_1,X_2,...,X_t)$ such that

$$A = \left\{ x \in \mathbb{R}^d: \Phi(f_1(x) \geq 0, f_2(x) \geq 0,...,f_t(x) \geq 0)=1\right\}.$$

\noindent We say that a semi-algebraic set has \emph{description complexity at most $t$} if $d\leq t$, the number of equations and inequalities is at most $t$, and each polynomial $f_i$ has degree at most $t$.

Let $F = \{A_1,...,A_N\}$ be an ordered family of semi-algebraic sets in $\mathbb{R}^d$ such that each set has bounded description complexity.  We denote ${F\choose k}$ to be the set of all \emph{ordered} $k$-tuples $(A_{i_1},...,A_{i_k})$ such that $i_1 < \cdots < i_k$.  Now let $E\subset{F\choose k}$ be a relation on $F$.  Typical examples of a relation $E$ would be $k$-tuples having a non-empty intersection, or having a hyperplane transversal, or having a clockwise orientation, etc. (see \cite{noga,basu}).  Since many of these relations can be described as a Boolean combination of a constant number of variables, we will assume that $E$ is semi-algebraic in the following sense. There exists a constant $q$ depending only on the description complexity so that each set $A_i \in F$ can be represented by a point $A^{\ast}_i$ in $\mathbb{R}^q$.  Then the relation $E$ is \emph{semi-algebraic with complexity at most $t$} if there exists a semi-algebraic set $E^{\ast} \subset \mathbb{R}^{qk}$ with description complexity $t$, where

$$E^{\ast} = \left\{ (x_1,...,x_k) \in \mathbb{R}^{qk}: \Phi(f_1(x_1,...,x_k) \geq 0, f_2(x_1,...,x_k) \geq 0,...,f_t(x_1,...,x_k) \geq 0)=1\right\}$$

\noindent for some polynomials $f_1,...,f_t$ of degree at most $t$ and Boolean function $\Phi$ and, for $i_1 <\cdots < i_k$, we have

\begin{equation}(A^{\ast}_{i_1},...,A^{\ast}_{i_k}) \in E^{\ast}\subset \mathbb{R}^{qk}\hspace{.4cm}  \Leftrightarrow\hspace{.4cm} (A_{i_1},...,A_{i_k}) \in E.\end{equation}

To simplify the presentation, we will always treat our semi-algebraic sets $A_1,..., A_N$ as points $p_1,..., p_N$ in a higher-dimensional space. Moreover, we will only consider {\em ordered} point sets $P = \{p_1,...,p_N\}$ in $\mathbb{R}^d$ with a semi-algebraic relation $E\subset {P\choose k}$. Note that the same results hold for {\em symmetric} relations on unordered point sets, as the ordering plays no role in this case.

We say that $(P,E)$ is \emph{homogeneous} if ${P\choose k}\subset E$ or $E\cap {P\choose k} = \emptyset$.  We denote by $\hom(P,E)$ the size of the largest 
homogeneous subfamily of $P$.  Let $R^{d, t}_{k}(n)$ be the minimum integer $N$ such that every ordered $N$-element point set $P$ in $\mathbb{R}^d$ equipped with 
a $k$-ary semi-algebraic relation $E\subset {P\choose k}$ which has complexity at most $t$ satisfies $\hom(P,E) \geq n$. Our first result shows that 
$R^{d,t}_{k}(n)$ may be bounded above by an exponential tower of height $k - 1$.

\begin{theorem}
\label{main}
For $k \geq 2$ and $d,t \geq 1$,
$$ R^{d, t}_{k}(n) \leq \twr_{k-1}\left(n^{c_1}\right),$$
\noindent where $c_1$ is a constant that depends only on $d$, $k$, and $t$.
\end{theorem}

\noindent We note that the $k = 2$ case of this result, which will prove crucial for our induction, is due to Alon, Pach, Pinchasi, Radoi\v ci\'c, and Sharir~\cite{noga}.

Adapting the stepping up approach of Erd\H os and Hajnal, we may also show that, for every $k$ and every sufficiently large $d$ and $t$, the function $R^{d, t}_{k}(n)$ does indeed grow as a $(k-1)$-fold exponential tower in $n$.

\begin{theorem}
\label{main2}
For every $k \geq 2$, there exist $d = d(k)$ and $t = t(k)$ such that

$$ R^{d,t}_{k}(n)\geq \twr_{k-1}(c_2n),$$

\noindent where $c_2$ is a positive constant that depends only on $k$.
\end{theorem}

\noindent Notice that in our proof we find it necessary to make $d$ large in terms of $k$. This is in some sense necessary since a striking result of Bukh and Matou\v sek \cite{bukh} implies that the one-dimensional semi-algebraic Ramsey function $R_{k}^{1,t} (n)$ is at most double exponential, that is, $R_{k}^{1,t} (n) \leq 2^{2^{cn}}$, where $c$ depends only on $k$ and $t$. Nevertheless, for $k = 4$, we will show that there is a one-dimensional construction giving the correct double-exponential lower bound. 

 \medskip

\noindent
{\bf Applications.} Over the past few decades, Ramsey numbers have been applied extensively to give upper bounds on homogeneity problems arising in geometry. For many of these applications, the relations can be defined algebraically and one can obtain an exponential improvement on the bound by using Theorem \ref{main}. Here we will present two such applications.

\medskip

\noindent \textit{Order types.}  Consider an ordered set $P= \{p_1,p_2,...,p_N\}$ of $N$ points in $\mathbb{R}^d$ in general position, that is, no $d+1$ members lie on a common hyperplane.  For a $(d + 1)$-tuple $ (p_{i_1},...,p_{i_{d + 1}})$, where $i_1 < \cdots < i_{d + 1}$, let $M = M(p_{i_1},...,p_{i_{d + 1}})$ be the $(d+ 1)\times (d+1)$ matrix with vectors $(1,p_{i_j})$, i.e., $1$ followed by the vector of the $d$ coordinates of $p_{i_j}$, as the columns for $1 \leq j \leq d + 1$ and let $\det(M)$ denote the determinant of the matrix $M$.  We say that $(p_{i_1},...,p_{i_{d + 1}})$ has a positive orientation if $\det(M) > 0$, and we say that $(p_{i_1},...,p_{i_{d + 1}})$ has a negative orientation if $\det(M) < 0$.  Notice that since $P$ is in general position $\det(M) \neq 0$.

Eli\'a\v{s} and Matou\v{s}ek \cite{matousek} defined $OT_d(n)$ to be the smallest integer $N$ such that any set of $N$ points in general position in 
$\mathbb{R}^d$ contains $n$ members such that every $(d + 1)$-tuple has the same orientation.  It was pointed out in \cite{matousek} that $OT_1(n) = (n - 1)^2 + 
1$, $OT_2(n) = 2^{\Theta(n)}$, and, for $d\geq 3$, the bound $OT_d(n) \leq \twr_{d+1}(cn)$ follows from Ramsey's theorem. They also gave a construction showing 
that $OT_3(n) \geq 2^{2^{\Omega(n)}}$.  In Section \ref{applications}, we prove the following result which improves the upper bound by one exponential.  In 
particular, for $d = 3$, it shows that the growth rate of $OT_3(n)$ is double exponential in $n^c$.

\begin{theorem}
\label{ot}
For $d\geq 3$, $OT_d(n) \leq \twr_{d}(n^c)$, where $c$ depends only on $d$.
\end{theorem}

\medskip

\noindent \textit{One-sided hyperplanes.}  Let $H=\{h_1,...,h_N\}$ be an ordered set of $N$ hyperplanes in $\mathbb{R}^d$ in general position, that is, every $d$ members in $H$ intersect at a distinct point.  We say that $H$ is \emph{one-sided} if the vertex set of the arrangement of $H$, that is, the set of intersection points, lies completely on one side of the hyperplane $x_d = 0$.

Let $OSH_d(n)$ denote the smallest integer $N$ such that every set of $N$ hyperplanes in $\mathbb{R}^d$ in general position contains $n$ members that are one-sided.  Dujmovi\'c and Langerman \cite{duj} used the existence of $OSH_d(n)$ to prove a ham-sandwich cut theorem for hyperplanes.  Matou\v{s}ek and Welzl \cite{welzl} observed that $OSH_2(n) = (n - 1)^2 + 1$ and Eli\'a\v{s} and Matou\v{s}ek \cite{matousek} noticed that, for $d\geq 3$, $OSH_d(n) \leq \twr_{d}(cn)$ follows from Ramsey's theorem, where $c$ depends on $d$.  In Section \ref{applications}, we prove the following result which again improves the upper bound by one exponential.

\begin{theorem}
\label{osh}
For $d\geq 3$, $OSH_{d}(n) \leq \twr_{d-1}(cn^2\log n)$, where $c$ depends only on $d$.

\end{theorem}

\medskip

\noindent
{\bf Off-diagonal Ramsey numbers for semi-algebraic relations.} 
The Ramsey number $R_k(s,n)$ is the minimum integer $N$ such that every red-blue coloring of the unordered $k$-tuples on an $N$-element set contains a red set of size $s$ or a blue set of size $n$, where a set is called red (blue) if all $k$-tuples from this set are red (blue).  The off-diagonal Ramsey numbers, i.e., $R_k(s,n)$ with $s$ fixed and $n$ tending to infinity, have been intensively studied. For example, it is known \cite{AKS,B,BK, Kim} that $R_2(3,n) =\Theta(n^2/\log n)$ and, for fixed $s > 3$,
\begin{equation}\label{offgraph}
c_1(\log n)^{1/(s-2)} \left(\frac{n}{\log n}\right)^{(s+1)/2} \leq R_2(s,n) \leq c_2\frac{n^{s-1}}{\log^{s-2}n}.\end{equation}

A classical argument of Erd\H os and Rado \cite{rado} (see also \cite{conlon} for an improvement) demonstrates that

\begin{equation}\label{rado}R_k(s,n) \leq 2^{{R_{k-1}(s-1,n-1)\choose k-1}} + k - 2.\end{equation}

\noindent Together with the upper bound in (\ref{offgraph}) this implies, for fixed $s$, that

\begin{equation}\label{rado2} R_k(s,n) \leq \twr_{k-1}(c_s n^{2s-2k+2}/\log^{2s-2k} n).
\end{equation}

\noindent The bound in \cite{conlon} roughly improves the exponent of $n$ from $2s-2k+2$ to $s-k+1$. 
Note that, for fixed 
$s$, we get an exponential improvement on the upper bound for $R_k(s,n)$ that follows from using the trivial bound $R_k(s,n) \leq R_k(n)$.

Off-diagonal Ramsey numbers may also be used to give another simple solution to the problem of estimating $ES(n)$. Let $P$ be a set of $N = R_4(5,n)$ points in 
the plane in general position.  We color the 4-element subsets blue if they are in convex position and color them red otherwise.  As noticed by Esther Klein, any five 
points in general position must contain four points in convex position.  Hence, there must be a subset $S\subset P$ of size $n$ such that every 4 points in $S$ is 
in convex position and therefore $S$ must be in convex position.  This shows that $ES(n) \leq R_4(5,n)$.  Just as before, one might expect that this 
double-exponential bound is not tight for such an algebraically defined coloring.

Let $R_{k}^{d,t}(s,n)$ denote the minimum integer $N$ with the property that for any sequence of $N$ points $P$ in $\mathbb{R}^d$ and any $k$-ary semi-algebraic relation $E\subset {P\choose k}$ of complexity at most $t$, $P$ has $s$ members such that every $k$-tuple induced by them is in $E$ or $P$ has $n$ members such that no $k$-tuple induced by them belongs to $E$.  Clearly, for $s \leq n$, we have

\begin{equation}\label{trivial}R_{k}^{d,t}(s,n) \leq R_{k}^{d,t}(n)\leq \twr_{k-1}(n^c),\end{equation}
which matches the tower height in (\ref{rado2}). However, it seems likely that the following stronger bound holds.

\begin{conjecture}

For fixed $k \geq 3,d,t$, and $s$,  $R_{k}^{d,t}(s,n) \leq \twr_{k-2}(n^c)$, where $c = c(k,d,t,s)$.

\end{conjecture}

  \noindent The crucial case is when $k = 3$, since a polynomial bound on $R_{3}^{d,t}(s,n)$ could be used with the adaptation of the  Erd\H os-Rado upper bound argument discussed in this paper to obtain an exponential improvement over the trivial bound in (\ref{trivial}) for all $k$.  In Section \ref{offdiag}, we prove a somewhat weaker result, giving a quasi-polynomial bound for point sets in one dimension.

\begin{theorem}
\label{off}
For $s \geq 4$, $R_{3}^{1,t}(s,n) \leq 2^{\log^c n}$, where $c$ depends only on $s$ and $t$.

\end{theorem}

\noindent Combining Theorem \ref{off} with our adaptation of the Erd\H os-Rado upper bound argument, we obtain the following result.

\begin{corollary}

For $n \geq 4$, $R_{4}^{1,t}(s,n) \leq 2^{2^{\log^c n}}$, where $c$ depends only on $s$ and $t$.

\end{corollary}

\noindent For $k\geq 5$, the result of Bukh and Matou\v{s}ek \cite{bukh} mentioned earlier implies that $R^{1,t}_{k}(s,n) \leq R^{1,t}_{k}(n) \leq 2^{2^{cn}}$, where $c = c(k,t)$.

For general $d$, we were only able to establish a good lower bound in the following special case. We say that the pair $(P,E)$ is {\em $K^{(3)}_4\setminus e$-free} if every {\em four} points induce at most {\em two} 3-tuples that belong to $E$.

\begin{theorem}
\label{david}
Let $P = \{p_1,...,p_N\}$ be a sequence of $N$ points in $\mathbb{R}^d$ with semi-algebraic relation $E\subset {P\choose 3}$ such that $E$ has complexity at most $t$.  If $(P,E)$ is $K^{(3)}_4\setminus e$-free, then there exists a subset $P'\subset P$ such that ${P'\choose 3} \cap E = \emptyset$ and

$$|P'| \geq 2^{\log^{c} N},$$

\noindent where $c < 1$ depends only on $d,t$.
\end{theorem}

\medskip

\noindent
{\bf Organization.} In the next section, we will prove Theorem \ref{main}, our upper bound on Ramsey numbers for semi-algebraic relations. Then, in Section \ref{lower}, we will prove the matching lower bound, Theorem \ref{main2}. We discuss the short proofs of our applications, Theorems \ref{ot} and \ref{osh}, in Section \ref{applications} and our results on the off-diagonal case, Theorems \ref{off} and \ref{david}, in Section \ref{offdiag}. We conclude with some further remarks.

\section{Upper bounds} \label{upper}

In this section, we prove Theorem \ref{main}.  First we briefly discuss the classic Erd\H os-Rado argument obtaining the recursive formula

$$R_k(n) \leq 2^{{R_{k-1}(n-1) \choose k-1}} + k - 2.$$

 Set $N = 2^{{R_{k-1}(n-1) \choose k-1}} + k - 2$ and $M = R_{k-1}(n-1)$.  Given a red-blue coloring $\chi$ on the $k$-tuples from $[N]$, Erd\H os and Rado greedily construct a sequence of distinct vertices $v_1,...,v_{M + 1}$ such that, for any given $(k-1)$-tuple $1 \leq i_1  < \cdots < i_{k - 1} \leq M$, all $k$-tuples $\{v_{i_1},...,v_{i_{k-1}},v_j\}$ with $j > i_{k-1}$ are of the same color, which  we denote by $\chi'(v_{i_1},...,v_{i_{k-1}})$.  Since $M = R_{k-1}(n-1)$, there is a monochromatic set of size $n-1$ in coloring $\chi'$. Together with the vertex $v_{M+1}$, these form a monochromatic clique of size $n$ in $\chi$.  The greedy construction of the sequence $v_1,...,v_{M + 1}$ is as follows.  First, pick $k-2$ arbitrary vertices $v_1,...,v_{k-2}$ and set $S_{k-2} = S\setminus\{v_1,...,v_{k-2}\}$.  After having picked $\{v_1,...,v_r\}$ we also have a subset $S_r$ such that for any $(k-1)$-tuple $v_{i_1},...,v_{i_{k-1}}$ with $1 \leq i_1 < \cdots < i_{k-1} \leq r$, all $k$-tuples $\{v_{i_1},...,v_{i_{k-1}},w\}$ with $w \in S_r$ are the same color.  Let $v_{
r+1}$ be an arbitrary vertex in $S_r$.  Let us call two elements $x,y \in S_r\setminus\{v_{r + 1}\}$ \emph{equivalent} if for every $(k-1)$-tuple $T \subset \{v_1,...,v_{r + 1}\}$ we have $\chi(T\cup\{x\}) = \chi(T\cup \{y\})$.  By the greedy construction, $x$ and $y$ are equivalent if and only if for every $(k-2)$-tuple $T\subset\{v_1,...,v_r\}$ we have $\chi(T\cup\{v_{r+1},x\}) = \chi(T\cup\{v_{r + 1},y\})$.  Therefore, there are ${r\choose k-2}$ possible choices for $T$ and hence there are at most $2^{r \choose k-2}$ equivalence classes.  We set $S_{r + 1}$ to be the largest of those classes.  Finally, we set $\chi'(v_{i_1},...,v_{i_{k-2}},v_{r+1}) = \chi(v_{i_1},...,v_{i_{k-2}},v_{r+1},w)$ where $w\in S_{r + 1}$.  As $N$ is large enough so that $S_{M}$ is nonempty, we can indeed construct the desired sequence of vertices.

There are two ways we improve the Erd\H os-Rado approach for semi-algebraic relations. Suppose that the $k$-tuples which are colored red under $\chi$ correspond to a semi-algebraic relation $E_1$ with bounded description complexity and let $E_2$ be the relation containing those $(k-1)$-tuples which are colored red by $\chi'$.  The main improvement comes from showing that $E_2$ will also be semi-algebraic with bounded description complexity.  Therefore, we can obtain by induction on $k$ an exponential improvement, starting with the result of Alon et al.~\cite{noga} as the base case $k = 2$.  A further improvement can be made by the observation that $S_{r}\setminus\{v_{r + 1}\}$ does not need to be partitioned into $2^{r \choose k-2}$ equivalence classes.  Instead, we can apply the Milnor-Thom Theorem (stated below) to partition $S_{r}\setminus\{v_{r + 1}\}$ into at most $O(r^{dk})$ equivalence classes with the desired properties, where the implied constant depends on the description complexity of $E_1$ and the uniformity $k$.

Let $f_1, ..., f_r$ be $d$-variate real polynomials with zero sets $Z_1, ..., Z_r$. A vector $\sigma \in \{-1, 0, +1\}^r$ is a {\it sign pattern} of $p_1, ..., p_r$ if there exists an $x \in \mathbb{R}^d$ such that the sign of $p_j(x)$ is $\sigma_j$ for all $j = 1,..., r$. The Milnor-Thom theorem (see \cite{basu, milnor, pet, thom}) bounds the number of cells in the arrangement of the zero sets $Z_1, ..., Z_r$ and, consequently, the number of possible sign patterns.

\begin{theorem}[Milnor-Thom]
Let $f_1,...,f_r$ be $d$-variate real polynomials of degree at most $D$.  The number of cells in the arrangement of their zero sets $Z_1,...,Z_r\subset \mathbb{R}^d$ and, consequently, the number of sign patterns of $f_1,...,f_r$ is at most

$$\left(\frac{50Dr}{d}\right)^d$$

\noindent for $r \geq d \geq 2$.
\end{theorem}

Theorem \ref{main} easily follows from the following recursive formula.

\begin{theorem}
\label{semirec}
Set $M = R^{d,t}_{k-1}(n-1)$.  Then, for every $k \geq 3$,

$$R^{d,t}_{k}(n)  \leq 2^{C_1M\log M},$$

\noindent where $C_1=C_1(k,d,t)$.
\end{theorem}

\begin{proof}
Let $P =\{p_1,...,p_N\}$ be a set of $N = 2^{C_1M\log M}$ points in $\mathbb{R}^d$ with semi-algebraic relation $E_1\subset {P\choose k}$ such that $E_1$ has 
complexity at most $t$ and where $C_1$ is a constant that will be specified later.  As mentioned earlier, we can represent $E_1$ by the set $E_1^{\ast} \subset 
\mathbb{R}^{dk}$ that satisfies $(1)$.  Since $E_1^{\ast}$ is semi-algebraic with complexity at most $t$, there exist polynomials $f_1,f_2,...,f_t \in 
\mathbb{R}[x_1,....,x_{dk}]$ of degree at most $t$ and a Boolean function $\Phi$ such that

$$E_1^{\ast} = \left\{ x \in \mathbb{R}^{dk}: \Phi\left(f_1(x) \geq 0,f_2(x) \geq 0,...,f_t(x) \geq 0\right)=1\right\}.$$

In what follows, we will recursively construct a sequence of points $q_1,...,q_r$ from $P$ and a subset $S_r \subset P$, where $r = k-2,k-1,...,M + 1$, such that the following holds. Every $(k-1)$-tuple $(q_{i_1},...,q_{i_{k-1}})\subset \{q_1,...,q_{r-1}\}$ with $i_1 < i_2 < \cdots <i_{k-1}$ has the property that either $(q_{i_1},...,q_{i_{k-1}}, q) \in E_1^{\ast}$ for every point $q \in \{q_j:i_{k-1} < j \leq r\} \cup S_r$ or $(q_{i_1},...,q_{i_{k-1}}, q) \not\in E_1^{\ast}$ for every point $q \in \{q_j:i_{k-1} < j \leq r\} \cup S_r$.  Also, $$|S_r| \geq \frac{N}{C_2^r  r^{dkr}}-r,$$ where $C_2=C_2(k,d,t)$ is a constant depending only on $k$, $d$, and $t$. Furthermore, for $i<j$, $q_i$ comes before $q_j$ in the original ordering and every point in $S_r$ comes after $q_r$ in the original ordering.

We start by selecting the $k-2$ points $\{q_1,...,q_{k-2}\}  = \{p_1,...,p_{k-2}\}$ from $P$ and setting $S_{k-2} = P \setminus\{p_1,...,p_{k-2}\}$.  After obtaining $\{q_1,...,q_r\}$ and $S_r$, we define $q_{r+1}$ and $S_{r+1}$ as follows.  Let $q_{r+1}$ be the smallest indexed element in $S_r$ and fix a $(k-2)$-tuple $(q_{i_1},q_{i_2},...,q_{i_{k-2}}) \subset \{q_1,...,q_r\}$.  Then, for each $j$ such that $1 \leq j \leq t$, we define $d$-variate polynomials $h_j \in \mathbb{R}[x_1,...,x_d]$ such that, for $x = (x_1,...,x_d) \in \mathbb{R}^d$,

$$h_j(x) = f_j(q_{i_1},q_{i_2},...,q_{i_{k-2}},q_{r+1},x).$$


After doing this for each $(k-2)$-tuple in $\{q_1,...,q_r\}$, we have generated at most $t{r \choose k-2}$ zero sets in $\mathbb{R}^d$.  By the Milnor-Thom theorem, the number of cells in the arrangement of these zero sets is at most $C_2r^{dk}$, where $C_2 = C_2(k,d,t)$.  By the pigeonhole principle, there exists a cell $\Delta\subset \mathbb{R}^{q}$ that contains at least $(|S_r| - 1)/C_2r^{dk}$ points of $S_r$.  The $h_i$ have the same sign pattern for each point in $\Delta$.  In other words, for any fixed $(k-1)$-tuple $(q_{i_1},...,q_{i_{k-1}}) \subset \{q_1,...,q_{r+1}\}$, we have either

$$(q_{i_1},...,q_{i_{k-1}},p_l) \in E_1^{\ast} \hspace{1cm}\forall p_l \in \Delta$$

\noindent or

$$(q_{i_1},...,q_{i_{k-1}},p_l) \not\in E_1^{\ast} \hspace{1cm}\forall p_l \in \Delta.$$

\noindent Let $S_{r+1}$ be the set of points $p_l$ in the cell $\Delta$.  Then we have the recursive formula

$$|S_{r+ 1}| \geq \frac{|S_r| -1 }{C_2r^{dk}}.$$

\noindent Substituting in the lower bound on $|S_r|$, we obtain the desired bound 

$$|S_{r+ 1}| \geq \frac{2^{C_1M\log M}}{(C_2)^{r+1}(r+1)^{dk(r+1)}}- (r+1).$$
This shows that we can construct the sequence $q_1,\ldots,q_{r+1}$ and the set $S_{r+1}$ with the desired properties. 

\noindent Since $C_1=C_1(k,d,t)$ is sufficiently large and $M = R^{d,t}_{k-1}(n-1)$, we have

$$|S_{M}| \geq 1.$$

\noindent Hence, $\{q_1,...,q_{M +1}\}$ is well defined for $M = R^{d,t}_{k-1}(n-1)$.  Set $F = \{q_1,...,q_{M}\}$.  We define the semi-algebraic set $E_2^{\ast}\subset \mathbb{R}^{q(k-1)}$ by

$$E_2^{\ast} = \{x \in \mathbb{R}^{q(k-1)}: \Phi(f_1(x,q_{M+1}) \geq 0,....,f_t(x,q_{M+1}) \geq 0)=1\},$$

\noindent and define the relation $E_2\subset {F\choose k-1}$ by

$$E_2 = \left\{(q_{i_1},...,q_{i_{k-1}}) \in {F\choose k-1}: (q_{i_1},...,q_{i_{k-1}}) \in E_2^{\ast}\right\}.$$

\noindent Therefore, $E_2$ is a semi-algebraic relation with complexity at most $t$.  
By the definition of the function $R^{d,t}_{k-1}(n-1)$, there exist $n-1$ 
points $\{q_{i_1},...,q_{i_{n-1}}\}\subset F$ such that every $(k-1)$-tuple belongs to $E_2$ or no such $(k-1)$-tuple belongs to $E_2$.  
Recall that by the construction of 
$F$, if $(q_{i_1},...,q_{i_{k-1}}) \in E_2^{\ast}$ then $(q_{i_1},...,q_{i_{k-1}}, q_j) \in E_1^{\ast}$ for every
$j>i_{k-1}$. Therefore
every $k$-tuple of $\{q_{i_1},...,q_{i_{n-1}}\} \cup \{q_{M+1}\}$ belongs to $E_1$ or no such $k$-tuple belongs to $E_1$.  Hence, the set 
$\{q_{i_1},...,q_{i_{n-1}}\} \cup \{q_{M+1}\}$ is homogeneous and this completes the proof. \end{proof}

\section{Lower bounds}
\label{lower}
In this section, we will prove Theorem \ref{main2}.  For every $n\geq 1$ and $k\geq 3$, we will construct an $N$-element point set $P_k(N)$ in $\mathbb{R}^{d}$ with semi-algebraic relation $E_k\subset {P_k(N)\choose k}$, where $N = \twr_{k-1}(n)$, $d = 2^{k-3}$, and $E_k^{\ast} \subset \mathbb{R}^{dk}$ is a semi-algebraic set with complexity at most $t = t(k)$ such that

$$\hom(P_{k}(N), E_{k}) =O(n),$$

\noindent where the implied constant depends only on $k$.

\subsection{Base case $k = 3$}
\label{k3}

\noindent Let $P_3(2^n) = \{1,2,...,2^{n}\}\subset \mathbb{R}$ with relation $E_3\subset {P_3(2^n)\choose 3}$, where

$$E_3^{\ast} = \left\{(x_1,x_2,x_3)\subset \mathbb{R}^3 : x_1 < x_2 < x_3,  x_1 + x_3 - 2x_2 \geq 0\right\}.$$

\noindent Clearly, $E_3^{\ast}$ has bounded complexity.  Now we claim that $\hom(P_3(2^n), E_3) \leq n+1$. Indeed, let $H= \{p_1,p_2,...,p_r\} \subset P_3(2^n)$ be a homogeneous subset such that $p_1 < p_2 < \cdots < p_r$.

 \medskip

 \noindent \emph{Case 1.}  Suppose that ${H\choose 3} \subset E_3.$  Then $p_{i+1} \geq 2p_i - p_1$ for all $1 \leq i \leq r-1$. 
Since $p_1 \geq 1$ and $p_2 \geq p_1+1$, we have by induction

$$p_i \geq 2^{i-2}+p_{1} \geq 2^{i-2}+ 1.$$

\noindent  Hence, if $r = n+2$, we have

$$2^{n} + 1 \leq p_r \leq 2^{n},$$

\noindent which is a contradiction.
\medskip

\noindent \emph{Case 2.}  Suppose that ${H\choose 3} \cap E_3 = \emptyset$.  Then, for $1 \leq  i < j < k\leq r$, we have

$$p_i + p_k - 2p_j < 0,$$

\noindent which implies

$$(2^n - p_i + 1) + (2^n -p_k + 1) -2(2^n - p_j + 1) > 0.$$

\noindent  Set $q_{r - j + 1} = 2^n - p_j + 1$ for $1\leq j \leq r$.  Then we have

$$1 \leq q_1 < q_2 < \cdots < q_r\leq 2^{n},$$

\noindent and $q_{i + 1} \geq 2q_{i} - q_1$.  By the same argument as above, the set $\{q_1,...,q_r\}$ must satisfy $r < n+2$.  This completes the proof.

\subsection{The Erd\H os-Hajnal stepping-up lemma}
\label{hajnal}
For $k\geq3$, we will adapt the Erd\H os-Hajnal stepping-up lemma to construct the point set $P_k(N)$ and the relation $E_k\subset {P_k(N)\choose k}$.  Before we describe this procedure, we will briefly sketch the classic Erd\H os-Hajnal stepping-up lemma (see also \cite{conlon},\cite{conlon2},\cite{graham}).

Let $k \geq 3$ and suppose that $E_1\subset{[N]\choose k}$ is a relation on $[N]$ such that $\hom([N],E_1) < n$.  The \emph{stepping-up lemma} uses $E_1$ to define a relation $E_2 \subset{[2^N]\choose k + 1}$ with the properties listed below.  We refer to $E_2$ as the step-up relation of $E_1$.

For any $a \in [2^N]$, write $a-1=\sum_{i=0}^{N-1}a(i)2^i$ with $a(i) \in \{0,1\}$ for each $i$. For $a \not = b$, let $\delta(a,b) = i+1$ where $i$ is the largest value for which $a(i) \not = b(i)$. Notice that

\begin{description}

\item[Property A:] $\delta(a,b) \not = \delta(b,c)$ for every triple $a<b<c$,

\item[Property B:] for $a_1 < \cdots < a_n$, $\delta(a_1,a_n) = \max_{1 \leq i \leq n-1}\delta(a_i,a_{i + 1})$.

\end{description}

Given any $(k+1)$-tuple $a_1<a_2<\cdots<a_{k+1}$ of $[2^N]$, consider the integers $\delta_i=\delta(a_i,a_{i+1}), 1\le i\le k$. If $\delta_1,\ldots,\delta_{k}$ form a monotone sequence, then let $(a_1,a_2,\ldots,a_{k+1})\in E_2$ if and only if $(\delta_1,\delta_2,\ldots,\delta_{k})\in E_1$.

Now we have to decide if the $(k+1)$-tuple $(a_1,\ldots,a_{k+1})\in E_2$ in the case when $\delta_1,\ldots,\delta_{k}$ is not monotone. We say that $i$ is a {\it local minimum} if $\delta_{i-1}>\delta_i<\delta_{i+1}$, a {\it local maximum} if $\delta_{i-1}<\delta_i>\delta_{i+1}$, and a \emph{local extremum} if it is either a local minimum or a local maximum.  This is well defined by Property A.  If $\delta_2$ is a local minimum, then set $(a_1,...,a_{k+1}) \in E_2$.  If $\delta_2$ is a local maximum, then set $(a_1,...,a_{k+1})\not\in E_2$.  All remaining edges will not be in $E_2$.

\smallskip

\begin{lemma}
\label{stepup}
If $\hom([N],E_1) < n$, then the step-up relation $E_2$ satisfies $\hom([2^N],E_2) < 2n + k - 4$.
   \end{lemma}

\begin{proof}
Suppose for contradiction that the set of vertices $S = \{a_1,...,a_{2n+k-4}\}$ is homogeneous with respect to $E_2$, where $a_1<\ldots < a_{2n+k-4}$.  Without loss of generality, we can assume that ${S\choose k+1 }\subset E_2$.  Set $\delta_i = \delta(a_i,a_{i + 1})$.

\medskip

\noindent \emph{Case 1.}  Suppose that there exists a $j$ such that $\delta_j,\delta_{j + 1},...,\delta_{j + n - 1}$ forms a monotone sequence. First assume that

$$\delta_j > \delta_{j + 1} > \cdots >\delta_{j + n - 1}.$$

\noindent Since $\hom([N],E_1) < n$, there exists a subsequence $j\leq i_1 < \cdots <i_k \leq j+n - 1$ such that $(\delta_{i_1},...,\delta_{i_k})\not\in E_1$. But then the $(k+1)$-tuple $(a_{i_1},...,a_{i_k},a_{i_k + 1})\not\in E_2$. Indeed, by Property B,

$$\delta(a_{i_h},a_{i_{h + 1}}) = \delta(a_{i_h},a_{i_{h} + 1}) = \delta_{i_h}.$$

\noindent
Therefore, since $\delta_{i_1},....,\delta_{i_k}$ form a monotone sequence and  $(\delta_{i_1},....,\delta_{i_k}) \not\in E_1$, we have that the $(k+1)$-tuple $(a_{i_1},...,a_{i_k},a_{i_k + 1})\not\in E_2$, contradicting that ${S\choose k+1 }\subset E_2$.  A similar argument holds if $\delta_j < \delta_{j + 1} < \cdots <\delta_{j + n - 1}.$

\medskip

\noindent \emph{Case 2.}  Since $\delta_1,\ldots,\delta_{2n-2}$ does not have a monotone subsequence of length $n$, it must have at least two local extrema.  Since between any two local minimums there must be a local maximum, this implies that there exists a local maximum among $\delta_1,\ldots,\delta_{2n-2}$ and, therefore, a $(k + 1)$-tuple not in $E_2$, contradicting our assumption.
\end{proof}

\subsection{Stepping up algebraically}

\noindent We will now adapt the Erd\H os-Hajnal stepping-up lemma to our semi-algebraic framework.  First we need some definitions.  For a point $p \in \mathbb{R}^d$, we let $B(p,\epsilon)$ be the closed ball in $\mathbb{R}^d$ of radius $\epsilon$ centered at $p$.   For any two points $p_1,p_2 \in \mathbb{R}^d$, where $p_1 = (a_{1,1},a_{1,2},...,a_{1,d})$ and $p_2 = (a_{2,1},a_{2,2},...,a_{2,d})$, we write $p_1\prec p_2$ if $a_{1,i} < a_{2,i}$ for all $1\leq i \leq d$.  We say that the set of points $\{p_1,...,p_N\}$ is \emph{increasing}, if

$$p_1\prec p_2\prec \cdots \prec p_N.$$

\noindent For $r > 0$, we say that $\{p_1,...,p_N\}$ is \emph{$\epsilon$-increasing} if, for $1 \leq i \leq N$, $q_i \in B(p_i,\epsilon)$ implies that

$$q_1 \prec q_2\prec\cdots \prec q_N.$$

\noindent
For any two points $q_1,q_2 \in \mathbb{R}^{2d}$, where

$$q_1 = (x_1,y_1,x_2,y_2,..., x_d,y_d)\hspace{.4cm}\textnormal{and}\hspace{.4cm} q_2 = (x'_1,y'_1,x'_2,y'_2,...,x'_d,y'_d),$$

\noindent we define the \emph{slope} $\sigma(q_1,q_2)$ of $q_1,q_2$ to be

$$\sigma(q_1,q_2) = \left(\frac{y'_1 - y_1}{x'_1 - x_1},...,\frac{y'_i - y_i}{x'_i - x_i},..., \frac{y'_d -y_d}{x'_d - x_d }\right) \in \mathbb{R}^d.$$

\noindent Thus, the $i$th coordinate of $\sigma(q_1,q_2)$ is the slope of the line through the points $(x_i,y_i)$ and $(x'_i,y'_i)$ in $\mathbb{R}^2$.

Let $P_k(N) = \{p_1,...,p_N\}$ be a set of $N$ points in $\mathbb{R}^d$ with relation $E_k \subset {P_k(N) \choose k}$ such that $E^{\ast}_{k}$ is a semi-algebraic set in $\mathbb{R}^{dk}$ with complexity at most $t$.  We say that $(P_k(N), E_k)$ is $\epsilon$-deep if moving any point in $P_k(N)$ by a distance at most $\epsilon$ will not change the relation $E_k$.  More precisely, $(P_k(N), E_k)$ is \emph{$\epsilon$-deep} if, for every $(p_{i_1},...,p_{i_k}) \in E_k$,

$$(q_{i_1},...,q_{i_k})\in E_k^{\ast}  \hspace{.4cm}\textnormal{for all}\hspace{.4cm} q_{i_1} \in B(p_{i_1},\epsilon), q_{i_2} \in B(p_{i_2},\epsilon), ....,q_{i_k} \in B(p_{i_k},\epsilon)$$

\noindent and, for every $(p_{i_1},...,p_{i_k}) \not\in E_k$,

$$(q_{i_1},...,q_{i_k})\not\in E_k^{\ast}  \hspace{.4cm}\textnormal{for all}\hspace{.4cm} q_{i_1} \in B(p_{i_1},\epsilon), q_{i_2} \in B(p_{i_2},\epsilon), ....,q_{i_k} \in B(p_{i_k},\epsilon).$$

\noindent
With these definitions in hand, our algebraic stepping-up lemma is now as follows.

\begin{lemma}[Stepping up]
\label{stepalg}
For $k \geq 3$ and $\epsilon > 0$,  let $P_k(N), E_k$, and $E_k^{\ast}$ be as above and such that $P_k(N)$ is $\epsilon$-increasing and $(P_k(N), E_k)$ is $\epsilon$-deep.  Then there exists $\epsilon_1 > 0$ such that there is an $\epsilon_1$-increasing point set $P_{k + 1}(2^N)$ of $2^N$ points in $\mathbb{R}^{2d}$ and a semi-algebraic set $E^{\ast}_{k + 1}$ in $\mathbb{R}^{2d(k+ 1)}$ with complexity $t_1 = t_1(t,k)$ for which $E_{k + 1} \subset{P_{k + 1}(2^N) \choose k + 1}$ is the step-up relation of $E_k$ and $(P_{k + 1}(2^N),E_{k + 1})$ is $\epsilon_1$-deep.
\end{lemma}

\noindent From Lemmas \ref{stepup} and \ref{stepalg}, we have the following immediate corollary.

\begin{corollary}
\label{smallhom}
$\hom(P_{k + 1}(2^N),E_{k + 1}) \leq 2\cdot \hom(P_{k }(N),E_{k }) + k - 4.$
\end{corollary}

\noindent \emph{Proof of Lemma \ref{stepalg}}: \textbf{Construction of $P_{k + 1}(2^N)$.}  Given an $\epsilon$-increasing point set $P_k(N) = \{p_1,...,p_N\}$ in $\mathbb{R}^d$, we will construct an $\epsilon_1$-increasing set $P_{k + 1}(2^N) = \{q_1,...,q_{2^N}\}$ of $2^N$ points in $\mathbb{R}^{2d}$ as follows.  For each $p_i \in P_k(N)$, we denote $p_i = (a_{i,1},...,a_{i,d})$.  The construction is done by induction on $N$.  For the base case $N = 1$, $p_1 = (a_{1,1},a_{1,2},...,a_{1,d})$ and $P_{k + 1}(2) = \{q_1,q_2\}$, where

$$q_1= (0,0,....,0) \hspace{.4cm}\textnormal{and}\hspace{.4cm} q_2 = (1,a_{1,1}, 1, a_{1,2},....,1,a_{1,d}).$$

\noindent For $N \geq 2$, set $B_1 = B((0,0,...,0),\epsilon_2)\subset \mathbb{R}^{2d}, B_2 = B((1,a_{N,1}, 1, a_{N,2},....,1,a_{N,d}), \epsilon_2) \subset \mathbb{R}^{2d}$, where $\epsilon_2$ is sufficiently small so that for any two points $q^{\ast}_1 \in B_1, q^{\ast}_2\in B_2$, we have

$$\sigma(q^{\ast}_1,q^{\ast}_2) \in B(p_N, \epsilon/2).$$

\noindent Given the $\epsilon$-increasing point set $P_k(N-1) = \{p_1,...,p_{N-1}\}\subset \mathbb{R}^d$, we inductively construct two small dilated copies of $P_{k + 1}(2^{N-1})$, $Q_1 = \{q_1,...,q_{2^{N-1}}\}$ and $Q_2 = \{q_{2^{N-1} + 1},...,q_{2^N}\}$, and translate them so that they lie inside $B_1$ and $B_2$ respectively.  Hence, the slope of any two points in $Q_i$ is preserved for $i \in \{1,2\}$.  Then $P_{k + 1}(2^N) = Q_1\cup Q_2$.  Since $P_k(N-1)$ is $\epsilon$-increasing, for $r\in \{1,2\}$, we have $Q_r$ is $\epsilon_1$-increasing for some $\epsilon_1> 0$.  Therefore, $P_{k + 1}(2^N)$ is $\epsilon_1$-increasing.

Now we make the following key observation on the point set $P_{k + 1}(2^N)$.

\begin{observation}
\label{key}

If $q_i \prec q_j$, then the point $\sigma(q_i,q_j) \in B(p_r,\epsilon/2)\subset \mathbb{R}^d$, where $r = \delta(i,j)$ and $\delta$ is defined as in Section \ref{hajnal}.
\end{observation}
\begin{proof}
This can be seen by induction on $N$.  For $i < j$, if $q_i \in Q_1, q_j \in Q_2$, then $\sigma(q_i,q_j) \in B(p_N,\epsilon/2)$ and $\delta(i,j) = N$.  If $q_i,q_j \in Q_1$ or $q_i,q_j \in Q_2$, then by the induction hypothesis and since the copies are slope preserving, we have

$$\sigma(q_i,q_j) \in B(p_r,\epsilon/2)$$

\noindent where $p_r \in \{p_1,...,p_{N-1}\}$ is such that $r = \delta(i,j)$.
\end{proof}

\noindent \textbf{Construction of $E^{\ast}_{k + 1}$.}  We define the semi-algebraic set $E^{\ast}_{k + 1} \subset \mathbb{R}^{2d(k + 1)}$ by

$$E^{\ast}_{k +1} = \left\{(x_1,...,x_{k + 1}) \in \mathbb{R}^{2d(k + 1)}: x_i \in \mathbb{R}^{2d},   x_{1} \prec x_{2} \prec \cdots \prec x_{k + 1}, C_1 \vee C_2 \vee C_3   \right\},$$

\noindent where conditions $C_1,C_2,C_3$ are defined below.

$$\begin{array}{cl}
    C_1: & \left[\sigma(x_{1},x_{2}) \succ \sigma(x_{2}, x_{3}) \right] \wedge \left[\sigma(x_{3}, x_{4}) \succ \sigma(x_{2}, x_{3})\right]. \\\\
    C_2: & \left[\sigma(x_{1},x_{2})\prec \sigma(x_{2},x_{3}) \prec \cdots \prec \sigma(x_{k},x_{k + 1}) \right]\wedge \left[(\sigma(x_{1},x_{2}),\sigma(x_{2},x_{3})...,\sigma(x_{k},x_{k+1})) \in E^{\ast}_{k}\right]. \\\\
    C_3: & \left[\sigma(x_{1},x_{2})\succ \sigma(x_{2},x_{3}) \succ \cdots \succ \sigma(x_{k},x_{k + 1}) \right]\wedge \left[(\sigma(x_{k},x_{k + 1}),...,\sigma(x_{2},x_{3}),\sigma(x_{1},x_{2})) \in E^{\ast}_{k}\right].\\\\
  \end{array}$$

\medskip

\noindent  Notice that $E_{k+1}$ is the step-up relation of $E_k$.  Indeed, let $(q_{i_1},...,q_{i_{k + 1}})$ be a $(k+1)$-tuple of points in $P_{k + 1}(2^N)$ such that $i_1 < \cdots < i_{k + 1}$.  If $\sigma(q_{i_1},q_{i_2}) \prec \cdots \prec \sigma(q_{i_k},q_{i_{k+1}})$, then

 $$(q_{i_1},...,q_{i_{k + 1}}) \in E^{\ast}_{k + 1} \Leftrightarrow (\sigma(q_{i_1},q_{i_2}),...,\sigma(q_{i_k},q_{i_{k + 1}})) \in E^{\ast}_k.$$

 \noindent  By Observation \ref{key} and since $(P_k(N),E_k)$ is $\epsilon$-deep, this happens if and only if $(p_{r_1},...,p_{r_k}) \in E_k$, where $r_l = \delta(i_l,i_{l + 1})$ for $1 \leq l \leq k$.  The same is true if  $\sigma(q_{i_1},q_{i_2}) \succ \cdots \succ \sigma(q_{i_k},q_{i_{k+1}})$.  If $\sigma(q_{i_1},q_{i_2}) \succ \sigma(q_{i_2},q_{i_3})\prec \sigma(q_{i_3},q_{i_4})$, then $(q_{i_1},...,q_{i_{k + 1}}) \in E^{\ast}_{k + 1}$ by condition $C_1$ and we have $\delta(i_1,i_2) > \delta(i_2,i_3) < \delta(i_3,i_4)$ by Observation \ref{key}.  Finally, if $(q_{i_1},...,q_{i_{k + 1}}) $ does not satisfy $C_1$, $C_2$, $C_3$, then $(q_{i_1},...,q_{i_{k + 1}}) \not\in E_{k + 1}^{\ast}$.

Although each coordinate of $\sigma$ is a rational function over 4 variables, by clearing denominators in the defining inequalities for $E^{\ast}_{k + 1}$, we get that $E^{\ast}_{k+1}$ is a semi-algebraic set with description complexity at most $c$ where $c = c(k,t)$.

\begin{observation}

$(P_{k + 1}(2^N),E_{k + 1})$ is $\epsilon_1$-deep.

\end{observation}

\begin{proof}

Suppose $(q_{i_1},q_{i_2},...,q_{i_{k + 1}}) \in E_{k + 1}$ with $q_{i_1} \prec\cdots \prec q_{i_{k + 1}}$ and let $q^{\ast}_i \in B(q_i,\epsilon_1)$ for $1\leq i \leq k +1$.  By Observation \ref{key}, $\sigma(q_{i_j},q_{i_{j + 1}}) \in B(p_l,\epsilon/2)$ for some $p_l \in P_k(N)$.  By making $\epsilon_1$ sufficiently small, we have $q^{\ast}_{i_1} \prec\cdots \prec q^{\ast}_{i_{k + 1}}$ and

$$\sigma(q^{\ast}_{i_j},q^{\ast}_{i_{j + 1}}) \in B(p_l,\epsilon).$$

\noindent Since $(P_k(N),E_k)$ is $\epsilon$-deep, $(q^{\ast}_{i_1},...,q^{\ast}_{i_{k + 1}}) \in E^{\ast}_{k + 1}$.  By a similar argument, if $(q_{i_1},...,q_{i_{k + 1}}) \not\in E_{k + 1}$, then $(q^{\ast}_{i_1},....,q^{\ast}_{i_{k+1}})\not\in E^{\ast}_{k + 1}$.
\end{proof}

Notice that our construction of $(P_3(2^n),E_3)$ is $(1/10)$-increasing, $(1/10)$-deep, and $E_3^{\ast}$ has constant description complexity. Applying Lemma \ref{stepalg} and Corollary \ref{smallhom} inductively on $k$ completes the proof of Theorem \ref{main2}.

\subsection{A construction in one dimension}
\label{k4}
A recent result of Bukh and Matou\v{s}ek \cite{bukh} shows that one can not keep stepping up in one dimension to another construction in one dimension.  Their result says that there exists a constant $c = c(k,t)$ such that $R^{1,t}_{k}(n) \leq 2^{2^{cn}}$.  They also showed that their result is tight by giving a matching lower bound in the case $k = 5$.  Here, we give another matching construction for the case $k = 4$.  This is the smallest possible uniformity for such a tight construction as, for $k \leq 3$, Theorem \ref{main} gives a single exponential upper bound.

The idea of the construction is to start with the point set $P_3(2^n)$ and relation $E_3\subset{P_3(2^n)\choose 3}$ from Subsection \ref{k3} and then to apply the Erd\H os-Hajnal stepping-up lemma with a sufficiently large base.  After stepping up, we obtain a point set $P$ with $2^{2^{n}}$ points in $\mathbb{R}$ with relation $E_3$ in the ``exponent".  Since the only operation in $E^{\ast}_3$ is addition, the step-up relation can be defined by multiplication and hence remains semi-algebraic.  We now formalize this idea.

\medskip

\noindent \emph{The construction:}  Let  $P_3(2^n)$ and $E^{\ast}_3$ be as in Section \ref{lower}. Recall that $(P_3(2^n), E_3)$ is $(1/10)$-deep.  We will step up by considering a point set $P$ on $2^{2^n}$ points, where for any $p\in P$ we have $p - 1 = \sum\limits_{i = 0}^{2^n - 1}p(i)b^i$ with $p(i) \in \{0,1\}$. Here $b$ is a sufficiently large constant to be determined later.  For $p>q$, let $\delta(p,q) = \log_b(p-q)$.  We will choose $b$ sufficiently large so that for any $p> q$, if $i$ is the largest integer such that $p(i) \neq q(i)$, then

$$i - \frac{1}{10} < \delta(p,q) < i + \frac{1}{10}.$$

\noindent Hence, $i$ is the closest integer to $\delta$ and this integer satisfies Properties A and B from Section \ref{hajnal}.  Now we define the relation $E\subset {P\choose 4}$ by the semi-algebraic set

$$E^{\ast} = \{(x_1,x_2,x_3,x_4) \in \mathbb{R}^4: x_1 < x_2 < x_3 < x_4,C_1\vee C_2 \vee C_3\},$$

\noindent where conditions $C_1,C_2,C_3$ are defined below.

$$\begin{array}{cl}
    C_1: & \left[\delta(x_1,x_2)  > \delta(x_2,x_3)\right] \wedge \left[\delta(x_3,x_4) > \delta(x_2,x_3)\right]. \\\\
    C_2: & \left[\delta(x_1,x_2) < \delta(x_2,x_3) < \delta(x_3,x_4) \right]\wedge \left[(\delta(x_1,x_2),\delta(x_2,x_3),\delta(x_3,x_4)) \in E^{\ast}_3\right]. \\\\
      C_3: & \left[\delta(x_1,x_2) > \delta(x_2,x_3) > \delta(x_3,x_4) \right]\wedge \left[(\delta(x_3,x_4),\delta(x_2,x_3),\delta(x_1,x_2)) \in E^{\ast}_3\right]. \\\\
  \end{array}$$

\noindent Notice that these conditions can be rewritten as

$$\begin{array}{cl}
    C_1: & \left[x_2 - x_1  > x_3 -x_2 \right] \wedge \left[x_4 - x_3 > x_3 - x_2\right]. \\\\
    C_2: & \left[x_2 - x_1 < x_3 - x_2 < x_4 - x_3 \right]\wedge \left[(x_2-x_1)(x_4-x_3) \geq (x_3 - x_2)^2\right]. \\\\
    C_3: & \left[x_2 - x_1 > x_3 - x_2 > x_4 - x_3 \right]\wedge \left[(x_2-x_1)(x_4-x_3) \geq (x_3 - x_2)^2\right].\\\\
  \end{array}$$

  \noindent Therefore, $E^{\ast}$ is semi-algebraic with constant description complexity. Using the fact that $(P_3,E_3)$ is (1/10)-deep, a similar argument to the previous subsection shows that $E^{\ast}$ is the step-up relation of $E^{\ast}_3$.  Corollary~\ref{smallhom} then implies that $(P,E)$ does not contain a homogeneous subset of size $2n + 1$.

\section{Applications}
\label{applications}

Let us recall that $OT_d(n)$ is the smallest integer $N$ such that any set of $N$ points in general position in $\mathbb{R}^d$ contains $n$ members such that every $(d + 1)$-tuple has the same orientation.  The proof of Theorem \ref{ot} giving an upper bound on $OT_d(n)$ follows quickly from Theorem \ref{main}.

\medskip
\noindent \emph{Proof of Theorem \ref{ot}.}
 Let $P= \{p_1,p_2,...,p_N\}$ be an ordered point set of $N$ points in $\mathbb{R}^d$ such that $P$ is in general position.  Let $E \subset {P\choose d + 1}$ be a relation on $P$ such that $(p_{i_1},...,p_{i_{d  +1}}) \in E$ if $(p_{i_1},...,p_{i_{d  +1}})$ has a positive orientation.  Then

 $$E^{\ast} = \{(x_1,...,x_{d+1}) \in \mathbb{R}^{d(d+1)}: x_i \in \mathbb{R}^d, \det(M(x_1,...,x_{d+1})) > 0\}.$$

 \noindent Thus, $E^{\ast}$ is a semi-algebraic set in $\mathbb{R}^{d(d + 1)}$ with description complexity at most $t= t(d)$.  Hence, the statement follows from Theorem \ref{main}. $\hfill\square$

\medskip

Recall that $OSH_d(n)$ is the smallest integer $N$ such that every set of $N$ hyperplanes in $\mathbb{R}^d$ in general position contains $n$ members that are one-sided, where a set of hyperplanes $H$ is one-sided if the vertex set of the arrangement of $H$ lies completely on one side of the hyperplane $x_d = 0$.  We obtain a stronger bound for Theorem \ref{osh} by deriving a recursive formula, similar to the one in Theorem \ref{semirec}.  Since each hyperplane $h_i\in H$ is specified by the linear equation

$$a_{i,1}x_1 + \cdots + a_{i,d}x_d = b_i,$$

\noindent we can represent $h_i\in H$ by the point $p_i \in\mathbb{R}^{d + 1}$ where $p_i = (a_{i,1},...,a_{i,d},b_i)$ and define a relation $E\subset {P\choose d}$.   However, for sake of clarity, we will simply define $E\subset {H\choose d}$, where $(h_{i_1},...,h_{i_d}) \in E$ if the point $h_{i_1}\cap\cdots\cap h_{i_d}$ lies above the hyperplane $x_d = 0$ (i.e. the $d$-th coordinate of the intersection point is positive).  Clearly, $E$ is a semi-algebraic relation with complexity at most $t = t(d)$.

Since $OSH_2(n) = (n-1)^2 + 1$, Theorem \ref{osh} follows immediately from the next theorem.

\begin{theorem}

For $d\geq 3$, let $M = OSH_{d-1}(n-1)$.  Then $OSH_d(n) \leq  2^{C_2M\log M}$, where $C_2$ depends only on $d$.

\end{theorem}

\begin{proof}

Set $N = 2^{C_2M\log M}$, where $C_2$ is a sufficiently large constant that depends only on $d$.  Let $H = \{h_1,...,h_N\}$ and $E\subset{H\choose d}$ be as defined above, and let $h_0$ be the hyperplane $x_d = 0$.  We now follow the proof of Theorem \ref{semirec} to the point where we obtain the sequence $q_1,...,q_{M+1} \in H$ such that every $(d-1)$-tuple $(q_{i_1},...,q_{i_{d-1}}) \subset \{q_1,...,q_{M}\}$ has the property that either $(q_{i_1},...,q_{i_{d-1}}, q_j)  \in E$ for all $i_{d-1} < j \leq M+1$ or $(q_{i_1},...,q_{i_{d-1}}, q_j) \not \in E$ for all $i_{d-1} < j \leq M+1$.

For each hyperplane $q_i \in \{q_1,...,q_{M}\}$, let $q^{\ast}_i = q_i\cap q_{M+1}$ and set $F = \{q^{\ast}_1,...,q^{\ast}_M\} \subset q_{M+1}$.  Hence, $F$ is a family of $M$ $(d-2)$-dimensional hyperplanes in $q_{M+1}$.  Since $q_{M+1}$ is isomorphic to $\mathbb{R}^{d-1}$ and $M = OSH_{d-1}(n-1)$, there exist $n-1$ members $F'=\{q^{\ast}_{i_1},...,q^{\ast}_{i_{n-1}}\} \subset F$ such that the vertex set of the arrangement of $F'$ lies completely on one side of the $(d-2)$-dimensional hyperplane $h_0\cap q_{M + 1}$ in $q_{M+1}$.  Let $H'\subset H$ be the hyperplanes corresponding to $F'$ in $\mathbb{R}^d$. By the construction of $H'\subset \{q_1,...,q_{M}\}$, the vertex set of the arrangement of $H' \cup \{q_{M+1}\}$ lies on one side of the hyperplane $h_0$.
\end{proof}

\section{The off-diagonal case}\label{offdiag}

In this section we prove Theorem \ref{off}, giving an upper bound on $R^{1,t}_{3}(s,n)$,  and Theorem \ref{david}.
We first list several results that we will use.

\begin{lemma}[Erd\H os-Szekeres \cite{erdos}]
\label{monotone}
Given a sequence of $N = (n-1)^2 + 1$ distinct real numbers $p_1,p_2,...,p_N$, there exists a subsequence $p_{i_1},p_{i_2},...,p_{i_n}$ of length $n$ such that either $p_{i_1} < p_{i_2} < \cdots < p_{i_n}$ or $p_{i_1} > p_{i_2} > \cdots > p_{i_n}$.

\end{lemma}

The next lemma is a combinatorial reformulation of another classical theorem due to Erd\H os and Szekeres \cite{erdos}.  A transitive 2-coloring of the triples of $[N]$ is a 2-coloring, say with colors red and blue, such that, for $i_1 < i_2 < i_3 < i_4$, if triples $(i_1,i_2,i_3)$ and $(i_2,i_3, i_4)$ are red (blue), then $(i_1,i_2,i_4)$ and $(i_1,i_3,i_4)$ are also red (blue).

\begin{lemma}[Fox et al.~\cite{fox}]
\label{fox}
Let $N_3(s,n)$ denote the minimum integer $N$ such that, for every transitive 2-coloring on the triples of $[N]$, there exists a red clique of size $s$ or a blue clique of size $n$.  Then

$$N_3(s,n)  = {s + n - 4\choose s-2} + 1.$$

\end{lemma}

\noindent The following lemma is the $k = 2$ case of Theorem \ref{main}, first proved by Alon, Pach, Pinchasi, Radoi\v ci\'c, and Sharir.

\begin{lemma}[Alon et al.~\cite{noga}]
\label{alon}
Let $P$ be a sequence of $N$ points in $\mathbb{R}^d$ and let $E\subset {P\choose 2}$ be a semi-algebraic relation on $P$ with description complexity at most $t$.  Then there exists a subset $P'\subset P$ with at least $N^{\alpha}$ elements such that either every pair of distinct elements of $P'$ belong to $E$ or no such pair belongs to $E$, where $\alpha > 0$ depends only on $t$ and $d$.
\end{lemma}

\noindent The following result, due to Fox, Gromov, Lafforgue, Naor, and Pach, tells us that if many triples of a point set $P$ satisfy a semi-algebraic relation $E$ then there is a large tripartite $3$-uniform hypergraph all of whose edges are in $E$. 

\begin{lemma}[Fox et al.~\cite{gromov}]
\label{gromov}
Let $P$ be a sequence of $N$ points in $\mathbb{R}^d$ and let $E\subset {P\choose 3}$ be a semi-algebraic relation on $P$ with description complexity at most $t$.   If $|E| \geq \epsilon {N\choose 3}$, then there exist disjoint subsets $P_1,P_2,P_3 \subset P$ such that $|P_i| \geq \epsilon^{c_3}N$ and, for all $p_1 \in P_1$, $p_2\in P_2$, and $p_3 \in P_3$, $(p_1,p_2,p_3) \in E$, where $c_3$ depends only on $t$ and $d$.
\end{lemma}

\noindent The following lemma of Spencer is now an exercise in \emph{The Probabilistic Method} (see \cite{alonspencer}).
\begin{lemma}[Spencer \cite{spencer}]
\label{spencer}
Let $H = (V,E)$ be a 3-uniform hypergraph on $N$ vertices.  If $E(H) \geq N/3$, then there exists a subset $S \subset V(H)$ such that $S$ is an independent set and

$$|S| \geq \frac{2N}{3}\left(\frac{N}{3|E(H)|}\right)^{1/2}.$$

\end{lemma}

\noindent The last lemma on our list is an old theorem due to Sturm (see \cite{basu}). Let $g(x)$ be a polynomial in $x$ with real coefficients. We say that the 
sequence of polynomials $g_0(x),g_1(x),...,g_t(x), g_{t+1}(x)=0$ is a \emph{Sturm sequence} for $g(x)$ if

$$g_0(x) = g(x), g_1(x) = g'(x), \hspace{.2cm}\textnormal{and}\hspace{.2cm}g_i(x)  = -\rem(g_{i-2},g_{i-1}) \hspace{.5cm}\textnormal{for}\hspace{.5cm}i \geq 2,$$

\noindent where $\rem(g_{i-2},g_{i-1})$ denotes the remainder of $g_{i-2}(x)/g_{i-1}(x)$.  Since the degrees of polynomials $g_i$ are strictly decreasing, 
there is always an index $t$ such that $g_{t+1}(x)=0$ and hence $g_{t-1}(x)  = q_{t}(x)g_t(x)$ for some non-zero 
polynomial $q_t(x)$.

\begin{lemma}[Sturm]
\label{sturm}
Let $g(x)$ be a polynomial in $x$ with real coefficients and  let $g_0(x),\ldots,g_t(x),$ $g_{t+1}(x)=0$ be the Sturm sequence for $g(x)$.  Suppose $g(a),g(b) 
\neq 0$ for $a < b$.  Then the number of distinct real roots of $g(x)$ in the open interval $(a,b)$ is $\sigma(a)-\sigma(b)$, where $\sigma(\xi)$ denotes the 
number of sign changes (ignoring zeros) in the sequence $g_0(\xi),g_1(\xi),....,g_t(\xi).$

\end{lemma}

Let $P = \{p_1,...,p_N\}$ be an ordered set of $N$ distinct real numbers.   By Lemma \ref{monotone}, one can always find a subset $P'\subset P$ of size $\sqrt{N}$ such that the elements of $P$ are either increasing or decreasing.  If necessary, by a change of variables we can assume that the elements of $P'$ are increasing.  Since this is a negligible loss, we will now only consider increasing point sets.

 Let $P = \{p_1,...,p_N\}$ be an increasing sequence of $N$ distinct real numbers and let $E\subset {P\choose 3}$ be a semi-algebraic relation on $P$ such that

$$E^{\ast}  = \{(x_1,x_2,x_3) \in \mathbb{R}^3: x_1 < x_2 < x_3, \Phi(f_1(x_1,x_2,x_3),...,f_t(x_1,x_2,x_3))=1\},$$

\noindent where the $f_i$ are polynomials of degree at most $t$ and $\Phi$ is a Boolean function.

The \emph{domain} of $P$ is the open interval $(p_1,p_N)$.  For each pair $p_i,p_j \in P$ with $i < j$, we write $\mathcal{P}(p_i,p_j)$ for the set of non-zero univariate polynomials

$$f_l(x_1,p_i,p_j),f_l(p_i,x_2,p_j),f_l(p_i,p_j,x_3),$$

\noindent for $1 \leq l \leq t$.  We say that $(P,E)$ has at most \emph{$r$ roots within its domain} if, for any pair $p_i,p_j \in P$, the univariate polynomials in $\mathcal{P}(p_i,p_j)$  have at most $r$ distinct real roots in total inside the interval $(p_1,p_N)$.  Note that $|\mathcal{P}(p_i,p_j)| \leq 3t$ and $r \leq 3t^2$.  We say that $(P,E)$ is \emph{$K^{(3)}_s$-free} if every collection of $s$ points in $P$ contains a triple not in $E$.  Theorem \ref{off} follows immediately from the observation above and the following theorem.

\begin{theorem}
\label{k4free}
Let $P = \{p_1,...,p_N\}$ be an increasing sequence of $N$ distinct points in $\mathbb{R}$ and let $E\subset {P\choose 3}$ be a semi-algebraic relation on $P$ such that $E$ has complexity at most $t$ and $(P,E)$ has at most $r$ roots within its domain.  If $(P,E)$ is $K^{(3)}_s$-free, then there exists a subfamily $P' \subset P$ such that ${P'\choose 3} \cap E = \emptyset$ and

$$|P'| \geq e^{\alpha^{\epsilon(r + s )}(\log N)^{\epsilon}},$$

\noindent where $0 < \epsilon,\alpha < 1$ depend only on $t$.

\end{theorem}

\begin{proof}

The proof is by induction on $N$, $r$, and $s$.  The base cases are $s = 3$, $r = 0$, or $N \leq (6t^2)^{2/\alpha}$.  When $N \leq (6t^2)^{2/\alpha}$, the statement holds trivially for sufficiently small $\epsilon$.  If $s = 3$, then again the statement follows immediately by taking $P' = P$.  For $r = 0$, notice that $E$ and $\overline{E} = {P\choose 3}\setminus E$ are both transitive.  Indeed, let $p_{i_1},p_{i_2},p_{i_3},p_{i_4} \in P$ be such that $p_{i_1} < \cdots < p_{i_4}$ and $(p_{i_1},p_{i_2},p_{i_3}),(p_{i_2},p_{i_3},p_{i_4}) \in E$.  Since the sign pattern of each univariate polynomial in $\mathcal{P}(p_{i_1},p_{i_2})\cup \mathcal{P}(p_{i_3},p_{i_4})$ does not change inside the interval $(p_1,p_N)$, this implies that $(p_{i_1},p_{i_2},p_{i_4}),(p_{i_1},p_{i_3},p_{i_4}) \in E$.  Likewise, if $(p_{i_1},p_{i_2},p_{i_3}),(p_{i_2},p_{i_3},p_{i_4}) \not\in E$, then we must have $(p_{i_1},p_{i_2},p_{i_4}),(p_{i_1},p_{i_3},p_{i_4}) \not\in E$.  Since $(P,E)$ is $K_s^{(3)}$-free, by Lemma \ref{fox}, there exists a subset $P_0\subset P$ such that $|P_0| \geq \Omega(N^{1/(s-2)})$ and ${P_0\choose 3}\cap E = \emptyset$.

Now assume that the statement holds if $r' \leq r$, $s' \leq s$, $N' \leq N$ and not all three inequalities are equalities.  Let $f_1,...,f_t\in \mathbb{R}[x_1,x_2,x_3]$ be polynomials of degree at most $t$ such that

$$E^{\ast}  = \{(x_1,x_2,x_3) \in \mathbb{R}^3: x_1 < x_2 < x_3, \Phi(f_1(x_1,x_2,x_3),...,f_t(x_1,x_2,x_3))=1\}.$$

By applying Lemma \ref{alon} twice, first fixing point $p_1$ and then fixing point $p_N$, there exists a subset $P_1\subset P$ of size $N^{\alpha_1}$, where $\alpha_1 > 0$ depends only on $t$, such that, for $j_1 < j_2$,
$$(p_1,p_{j_1},p_{j_2}) \in E  \textnormal{ for all }  p_{j_1},p_{j_2} \in P_1\setminus\{p_1\} \hspace{.5cm}\textnormal{ or }\hspace{.5cm} (p_1,p_{j_1},p_{j_2}) \not\in E \textnormal{ for all } p_{j_1},p_{j_2} \in P_1\setminus\{p_1\},$$

\noindent and
$$(p_{j_1},p_{j_2},p_N) \in E  \textnormal{ for all }  p_{j_1},p_{j_2} \in P_1\setminus\{p_N\} \hspace{.5cm}\textnormal{ or }\hspace{.5cm} (p_{j_1},p_{j_2},p_N) \not\in E \textnormal{ for all } p_{j_1},p_{j_2} \in P_1\setminus\{p_N\}.$$

\noindent We call an ordered triple $(p_i,p_j,p_m)$ \emph{bad} if there exists a polynomial $f \in \mathcal{P}(p_i,p_j)$ such that $f$ has a root at $p_m$ or if there exists a polynomial $f \in \mathcal{P}(p_j,p_m)$ such that $f$ has a root at $p_i$. Since $\mathcal{P}(p_i,p_j)\cup \mathcal{P}(p_j,p_m)$  gives rise to at most $6t^2$ distinct roots, $P_1$ has at most $3t^2|P_1|^2$ bad triples.  By Lemma \ref{spencer}, there exists a subset $P_2 \subset P_1$ such that $|P_2| \geq N^{\alpha_2}$, where $\alpha_2 > 0$ depends only on $t$, and $P_2$ has no bad triple.

Let $P_2 = \{q_1,...,q_{N^{\alpha_2}}\}$ with $q_1 < \cdots < q_{N^{\alpha_2}}$.  We now partition $P_2 = Q_1\cup \cdots \cup Q_{M}$ into $M = 
\sqrt{N^{\alpha_2}}$ parts, such that

$$Q_i = \left\{q_j : (i-1)\sqrt{N^{\alpha_2}} + 1 \leq j \leq  i\sqrt{N^{\alpha_2}}\right\}.$$

\noindent  Let $I_i$ be the domain of $Q_i$.  We define the relation $E_i\subset {Q_i\choose 2}$ on $Q_i$, where $(q_{j_1},q_{j_2}) \in E_i$ if the non-zero univariate polynomials in $\mathcal{P}(q_{j_1},q_{j_2})$ have (in total) strictly less than $r$ roots inside the open interval $I_i$.

For $l \in \{1,...,t\}$, the Euclidean Algorithm implies that the univariate polynomial $f_l(x_1,x_2,x_3)$ in $x_1$ has a Sturm sequence of length at most $t$.  The same is true for the univariate polynomials $f_l(x_1,x_2,x_3)$ in $x_2$ and $x_3$.  Since there are no bad triples in $P_2$, we can apply Sturm's Lemma \ref{sturm}, which tells us that $E_i$ depends only on the polynomials $f_1,...,f_t$, their Sturm sequences, the endpoints of $I_i$, and $r \leq 6t^2$.  Hence, $E_i$ is semi-algebraic with complexity at most $t' = t'(t)$.  By Lemma \ref{alon}, there exists $S_i \subset Q_i$ such that $|S_i| \geq N^{\alpha}$, where $\alpha > 0$ depends only on $t$, and either

$${S_i\choose 2} \subset E_i \hspace{.5cm}\textnormal{or}\hspace{.5cm} {S_i\choose 2}\cap E_i = \emptyset.$$

\noindent We may assume that $\alpha < \alpha_2/24$.  If ${S_i\choose 2} \subset E_i$ for some $i$, then $(S_i, E)$ has at most $r-1$ roots within its domain.  By the induction hypothesis, there exists a subset $P_3 \subset S_i$ such that ${P_3\choose 3}\cap E = \emptyset$ and

$$|P_3| \geq e^{\alpha^{\epsilon  (r-1 + s)}(\log N^{\alpha})^{\epsilon}} = e^{\alpha^{\epsilon(r+ s)}(\log N)^{\epsilon}},$$

\noindent and we are done.  Therefore, we can assume that ${S_i\choose 2}\cap E_i = \emptyset$ for all $i$.  Hence, for any $q_{j_1},q_{j_2}\in S_i$ with $j_1 < j_2$, all $r$ roots (within the interval $(p_1,p_N)$) of the univariate polynomials in $\mathcal{P}(q_{j_1},q_{j_2})$ lie inside $I_i$.  Now we make the following observation.

\begin{observation}
For any two parts $S_{i_1}$ and $S_{i_2}$, where $i_1 < i_2$, either $(q_{j_1},q_{j_2},q_{j_3}) \in E$ for all $q_{j_1},q_{j_2} \in S_{i_1}$ and $q_{j_3}\in S_{i_2}$ or $(q_{j_1},q_{j_2},q_{j_3})\not\in E$ for all $q_{j_1},q_{j_2} \in S_{i_1}$ and $q_{j_3} \in S_{i_2}$.  Likewise, either $(q_{j_1},q_{j_2},q_{j_3}) \in E$ for all $q_{j_1} \in S_{i_1}$ and $q_{j_2},q_{j_3}\in S_{i_2}$ or $(q_{j_1},q_{j_2},q_{j_3})\not\in E$ for all $q_{j_1}\in S_{i_1}$ and $q_{j_2}, q_{j_3} \in S_{i_2}$.

\end{observation}

\begin{proof}

We first prove the first part of the statement.  Since all $r$ roots of the non-zero univariate polynomials in $\mathcal{P}(q_{j_1},q_{j_2})$ lie inside the 
interval $I_{i_1}$ and $q_{j_3}, p_N$ lie to the left of $I_{i_1}$, we have that  $f_l(q_{j_1},q_{j_2},q_{j_3})$ and $f_l(q_{j_1},q_{j_2},p_N)$ have the same 
sign for all $1 \leq l \leq t$. Therefore  

$$(q_{j_1},q_{j_2},q_{j_3})\in E \Leftrightarrow (q_{j_1},q_{j_2},p_N)\in E.$$

\noindent Since our sets are subsets of $P_1$, we have either $(q_{j_1},q_{j_2},q_{j_3}) \in E$ for all $q_{j_1},q_{j_2} \in S_{i_1}$ and $q_{j_3} \in S_{i_2}$ or $(q_{j_1},q_{j_2},q_{j_3}) \not\in E$ for all $q_{j_1},q_{j_2} \in S_{i_1}$ and $q_{j_3} \in S_{i_2}$.  The second part of the statement follows by the same argument.
\end{proof}

If there exist two parts $S_{i_1},S_{i_2}$ with $i_1 < i_2$ such that $(q_{j_1},q_{j_2},q_{j_3}) \in E$ for all $q_{j_1},q_{j_2} \in S_{i_1}$ and $q_{j_3} \in S_{i_2}$, then $(S_{i_1},E)$ is $K_{s-1}^{(3)}$-free.  By the induction hypothesis, there exists a subset $P_4\subset S_{i_1}$ such that ${P_4\choose 3}\cap E = \emptyset$ and
$$|P_4| \geq e^{\alpha^{\epsilon  (r + s-1 )}(\log N^{\alpha})^{\epsilon}} =  e^{\alpha^{\epsilon( r  +s)}(\log N)^{\epsilon}}.$$

\noindent The same is true if $(q_{j_1},q_{j_2},q_{j_3}) \in E$ for all $q_{j_1} \in S_{i_1}$ and $q_{j_2},q_{j_3} \in S_{i_2}$.  Therefore, we can assume that for $i_1 < i_2$ and $j_1 < j_2 < j_3$,
\begin{equation} \label{21}
(q_{j_1},q_{j_2},q_{j_3}) \not\in E \hspace{.5cm}\textnormal{for all}\hspace{.5cm}  q_{j_1},q_{j_2} \in S_{i_1},\hspace{.2cm} q_{j_3} \in S_{i_2} \hspace{.5cm}\textnormal{and for all}\hspace{.5cm}  q_{j_1} \in S_{i_1},\hspace{.2cm} q_{j_2},q_{j_3} \in S_{i_2}.
\end{equation}

Set $S = S_1\cup \cdots \cup S_M$ and recall that $M = \sqrt{N^{\alpha_2}}$ and $|S_i| = N^{\alpha}$.  For $i_1 < i_2$, let $q_{j_1}\in S_{i_1}$ and $q_{j_2} \in S_{i_2}$.  Then we say that the unordered triple $(q_{j_1},q_{j_2},S_i)$ is \emph{homogeneous} if

\begin{enumerate}
 \item for $i > i_2$, $(q_{j_1},q_{j_2},q_i) \in E$ for all $q_i \in S_i$ or $(q_{j_1},q_{j_2},q_i) \not\in E$ for all $q_i \in S_i$, or

 \item for $i_1 < i < i_2$, $(q_{j_1},q_i,q_{j_2}) \in E$ for all $q_i \in S_i$ or $(q_{j_1},q_i, q_{j_2}) \not\in E$ for all $q_i \in S_i$, or

  \item for $i < i_1  < i_2$, $(q_i,q_{j_1},q_{j_2}) \in E$ for all $q_i \in S_i$ or $(q_i,q_{j_1}, q_{j_2}) \not\in E$ for all $q_i \in S_i$.
\end{enumerate}

\noindent 
Since $q_{j_1},q_{j_2}$ give rise to at most $3t$ polynomials of degree at most $t$, there are at most $3t^2$ sets $S_i$ such that $(q_{j_1},q_{j_2},S_i)$ is not 
homogeneous. Indeed, the domain of every such $S_i$ must contain a root of one of these polynomials and the total number of roots is at most $3t^2$. 

We pick $b$ distinct members of the collection $\{S_1,...,S_M\}$ uniformly at random.  Let $X$ denote the number of non-homogeneous triples 
$(q_{j_1},q_{j_2},S_i)$, where $S_i$ is a set from our randomly chosen collection and $q_{j_1}$ and $q_{j_2}$ also lie in distinct sets from this collection.  
Since $|S_i|=N^{\alpha}$, then

$$\mathbb{E}[X] \leq {M\choose 2}\left(N^{\alpha}\right)^2 3t^2 \left(\frac{b}{M}\right)^3\leq \frac{2t^2b^3N^{2\alpha}}{M}.$$

\noindent By setting $b = M^{1/9}$ and since $\alpha < \alpha_2/24$, we have

$$\mathbb{E}[X] \leq \frac{2t^2}{M^{1/2}} = \frac{2t^2}{N^{\alpha_2/4}} \leq \frac{2t^2}{N^{6\alpha}}.$$

\noindent   Since  $N > (6t^2)^{2/\alpha}$, we have $\mathbb{E}[X] < 1$.  Hence, there exists a subset $T\subset S$ such that $T = T_1\cup\cdots \cup T_b$ where $b = N^{\alpha_2/18}$, $|T_i| = N^{\alpha}$, and, for any $q_{j_1}, q_{j_2}$ from distinct subsets, $(q_{j_1},q_{j_2},T_i)$ is homogeneous.  Therefore, we obtain the following.

\begin{observation}
\label{111}
For parts $T_{i_1}, T_{i_2}, T_{i_3}$, where $i_1 < i_2 < i_3$, either $(q_{j_1},q_{j_2},q_{j_3}) \in E$ for all $q_{j_1} \in T_{i_1}, q_{j_2} \in T_{i_2}, q_{j_3} \in T_{i_3}$ or $(q_{j_1},q_{j_2},q_{j_3}) \not\in E$ for all $q_{j_1} \in T_{i_1}, q_{j_2} \in T_{i_2}, q_{j_3} \in T_{i_3}$.

\end{observation}

\begin{proof}
Let $q_{j_1} \in T_{i_1}, q_{j_2} \in T_{i_2}$, and $q_{j_3} \in T_{i_3}$ be such that $(q_{j_1},q_{j_2},q_{j_3}) \in E$.  It suffices to show that for  $a_{1} \in T_{i_1}, a_{2} \in T_{i_2}$, and $a_{3} \in T_{i_3}$ we have $(a_1,a_2,a_3) \in E$.  Since $(q_{j_1},q_{j_2},T_{i_3})$ is homogeneous, we have $(q_{j_1},q_{j_2},a_3) \in E$.  Likewise, since $(q_{j_1},T_{i_2},a_3)$ is homogeneous, we have $(q_{j_1},a_2,a_3) \in E$.  Finally, since $(T_{i_1},a_2,a_3)$ is homogeneous, we have $(a_1,a_2,a_3) \in E$.   If $(q_{j_1},q_{j_2},q_{j_3}) \not\in E$, then by the same argument we have $(a_1,a_2,a_3) \not\in E$.
\end{proof}

Let $T'$ be a point set formed by selecting one point from each $T_i$.  Then, by applying the induction hypothesis on $(T',E)$ 
we can find a set of indices $J$ such that $|J| \geq e^{\alpha^{\epsilon( r + s )}(\log N^{\alpha_2/18})^{\epsilon}}$ and for every $j_1<j_2<j_3$ in $J$
all triples with one vertex in each $T_{j_i}$ do not satisfy $E$. Applying the induction hypothesis again to every $(T_j, E), j \in J$ we obtain a collection of 
subsets $U_j \subset T_j, |U_j|\geq e^{\alpha^{\epsilon( r + s )}(\log N^{\alpha})^{\epsilon}}$ such that ${U_j \choose 3} \cap E =\emptyset$. Let 
$P'=\cup_{j \in J}U_j$. Then, by 
(\ref{21}), this subset satisfies ${P' \choose 3} \cap E =\emptyset$. Moreover,

\begin{eqnarray*}
|P'| & \geq & e^{\alpha^{\epsilon( r + s )}(\log N^{\alpha_2/18})^{\epsilon}}e^{\alpha^{\epsilon( r + s )}(\log N^{\alpha})^{\epsilon}}\\\\
 & \geq & e^{\alpha^{\epsilon( r + s )}(\log N^{\alpha})^{\epsilon}}e^{\alpha^{\epsilon( r + s )}(\log N^{\alpha})^{\epsilon}}\\\\
  & \geq &  e^{\alpha^{\epsilon( r + s )}(\log N)^{\epsilon}},
\end{eqnarray*}

\noindent for sufficiently small $\epsilon = \epsilon(t)$ such that $2\alpha^{\epsilon}\geq 1$.
\end{proof}

\noindent \emph{Proof of Theorem \ref{david}:}  The proof is by induction on $N$.   The base case $N\leq 2^{6c_3}$ is trivial for sufficiently small $c = c(d,t)$, where $c_3$ is the constant from Lemma \ref{gromov}.  Furthermore, $c$ will be sufficiently small so that for all $N > 2^{6c_3}$,

\begin{equation}
\label{half}
\left(\log N - 3c_3\log^{c} N\right)^{c} \geq \log^{c} N  - 1.\end{equation}

    \noindent Set $\alpha = 2^{-3\log^{c} N}$. Now assume that the statement holds for all $N' < N$, where $N > 2^{6c_3}$.  The proof falls into two cases.

\medskip

\noindent \emph{Case 1.} Suppose $|E| \leq \alpha{N\choose 3}$.  Then, by Lemma \ref{spencer}, there exists $P'\subset P$ such that $P'\cap E = \emptyset$ and

$$|P'| \geq \frac{2}{3} N\left(\frac{N   }{3 \alpha{N\choose 3}}\right)^{1/2} \geq \frac{2\sqrt{2}}{3}2^{(3/2)\log^{c} N} \geq 2^{\log^{c} N},$$

\noindent so we are done.

\medskip

\noindent \emph{Case 2.}  Suppose $|E| > \alpha{N\choose 3}$.  By Lemma \ref{gromov}, there exist disjoint subsets $P_1,P_2,P_3 \subset P$ such that $|P_i| \geq \alpha^{c_3}N$ and, for every $p_1 \in P_1$, $p_2 \in P_2$, and $p_3 \in P_3$, we have $(p_1,p_2,p_3) \in E$.  Since $(P,E)$ is $K^{(3)}_4\setminus e$-free, we have $(p_1,p_2,p_3) \not\in E$ for every $p_1,p_2 \in P_1$ and $p_3 \in P_2$ and, likewise, for every $p_1 \in P_1$ and $p_2,p_3 \in P_2$.  By applying the induction hypothesis on $(P_1,E)$ and $(P_2,E)$, there exists $P'\subset P_1\cup P_2$ such that $P'\cap E = \emptyset$ and

$$|P'| \geq  2\cdot \left(2^{ \log^{c}(\alpha^{c_3}N)}\right) = 2^{1 + \left(\log N - 3c_3 \log^{c} N\right)^{c}}.$$

\noindent By (\ref{half}), we have $|P'| \geq 2^{\log^{c} N}$. $\hfill\square$

\section{Concluding remarks}

1.  We showed that for every integer $k\geq 2$ there exist $d = d(k)$ and $t = t(k)$ such that $R^{d,t}_{k}(n) \geq \twr_{k-1}(c_2 n)$, where $c_2$ depends on $k$.  Our construction gives dimension $d = 2^{k-3}$.  Instead of starting from the construction in Subsection \ref{k3}, we could start from the construction in Subsection \ref{k4} to get dimension $d = 2^{k-4}$ for $k \geq 4$.  It would be interesting to see if one could give constructions for all $k\geq 5$ in, say, two dimensions. The result of Bukh and Matou\v{s}ek mentioned in the introduction, that $R^{1,t}_{k}(n)$ is at most double exponential for fixed $k$ and $t$, shows that this cannot be done in one dimension.

\medskip

\noindent 2.  It would be very interesting if one could improve the bounds in the off-diagonal case.  The crucial case is when $k = 3$ and we conjecture that $R^{d,t}_{3}(s,n) = O(n^c)$, where $c$ depends only on $d, t$ and $s$.  Notice that this would give another proof that $ES(n)$ is at most exponential in a power of $n$.  Indeed, by the argument discussed in the introduction, there exists a $t$ such that $ES(n) \leq R^{2,t}_{4}(5,n)$.  A careful analysis of the proof of Theorem~\ref{semirec} gives $R^{d,t}_{k}(s,n) \leq 2^{C_1M\log M}$, where $M = R^{d,t}_{k-1}(s-1,n-1)$ and $C_1 = C_1(d,k,t)$.  Therefore, $R^{d,t}_{3}(4,n) = O(n^c)$ implies that

 $$ES(n) \leq  R^{2,t}_{4}(5,n) \leq 2^{C_2n^c\log n},$$

 \noindent where $C_2 = C_2(t)$.

\medskip
\noindent 3. A lower bound for the off-diagonal case may also be achieved by using the semi-algebraic version of the stepping-up lemma. In particular, we may show that for every integer $k \geq 3$ there exist $d = d(k), t = t(k)$, and $s = s(k)$ such that $R^{d,t}_k(s,n) \geq \twr_{k-2} (c' n)$, where $c'$ depends only on $k$. The original stepping-up lemma gives $s(k) \leq 2^{k-1} - k + 3$ and a recent variant \cite{conlon2} gives the linear bound $s(k) \leq \lceil \frac{5}{2} k\rceil - 3$ for $k \geq 4$. However, we conjecture that the result should already be true with $s = k+1$. 

\medskip

\noindent 4.  Let us remark that our results and their proofs generalize to multiple relations (colors).  More precisely, let $R^{d,t}_{k,q}(n)$ be the least integer $N$ such that any set of $N$ points in $\mathbb{R}^d$ with semi-algebraic relations $E_1,...,E_q\subset {P\choose k}$, where each $E_i$ has complexity at most $t$, contains $n$ members such that every $k$-tuple belongs to $E_i$ for some fixed $i$ or no such $k$-tuple belongs to $E_i$ for all $i$.  Then

$$R_{k,q}^{d,t}(n) \leq \twr_{k-1}(n^c),$$

\noindent where $c = c(d,k,t,q)$.  In particular, for $q \geq 3$, corresponding to $q + 1 \geq 4$ colors, this is a much smaller bound than in the general case, as $R_k(n,n,n,n) = \twr_{k}(\Theta(n))$, where $R_k(n,n,n,n)$ is the 4-color Ramsey number. 

\medskip

\noindent 5. One can also prove the following strengthening of Theorem \ref{main2} for four or more colors. For any $k \geq 2$ and $\alpha > 0$, there 
exist $d = d(k, \alpha)$, $t = t(k, \alpha)$, and $c' = c'(k, \alpha)$ such that, for $q \geq 3$, 
$R^{d,t}_{k,q}(n) \geq \twr_{k-1}(c' n^{\alpha})$. The key extra tool is a variant of the stepping-up lemma (see \cite{graham}) which allows one to step up from 
$k = 2$ to $k = 3$ at the cost of doubling the number of colors. For our base case, we need to know that for any $\alpha > 0$ there exist $d$ and $t$ such that, 
for $n$ sufficiently large, $R_2^{d,t}(n) \geq n^\alpha$. This follows from the following famous construction of Frankl and Wilson \cite{frankl}. Let $p$ be a prime 
and let $r=p^2-1$. The vertices of our graph $G$ are all subsets of size $r$ from $\{1, \ldots, m\}$ and two vertices are adjacent if the corresponding sets have 
intersection of size congruent to $-1\ (\mbox{mod } p)$. Let $n={m \choose p - 1}$. It was shown in \cite{frankl} that this graph has ${m \choose r}=\Omega(n^{p+1})$ 
vertices and no clique or independent set of size larger than $n$. This graph can be realized in dimension $d=r$ using a semi-algebraic 
relation of description complexity $t=r^2$. Every vertex of $G$ will correspond to a vector $(x_1, x_2, \ldots, x_r)$ where $x_1 < x_2 < \cdots < x_r$ are the elements of the 
corresponding subset of $\{1, \ldots, m\}$ in increasing order. For any two vertices $(x_1, \ldots, x_r)$ and $(y_1, \ldots, y_r)$ the Boolean function will check whether 
the number of equalities $x_i=y_j, 1 \leq i,j \leq r$ (i.e, the size of the intersection) is congruent to $-1 \ (\mbox{mod } p)$. 
The lower bound on $R^{d,t}_{k,q}(n)$ now follows by first stepping-up from $k = 2$ to 
$k=3$, thereby doubling the number of colors to $4$, and then applying the method of Section \ref{lower} for higher uniformities, which does not increase the 
number of colors. It would be interesting to know whether a similar improved result holds for the $2$-color case, $q = 1$.

\medskip
\noindent 6. For $k \geq 4$, the upper bound in Lemma \ref{stepup} and Corollary \ref{smallhom} can be improved by roughly a factor of $2$ using a slight variation of the stepping-up lemma from \cite{conlon2}.  In particular, starting with a $4$-uniform construction of order $N$ with no homogeneous set of order $n$, such as the one in Subsection \ref{k4}, we get a construction of order $\twr_{k-3}(N)$ with no homogeneous set of order $n+\lceil 5k/2 \rceil -10$.  This follows from Theorem 3 in \cite{conlon2}.

\end{document}